\documentclass{article}

\usepackage{amsmath}
\usepackage{amsthm}
\usepackage{amsfonts}
\usepackage{graphicx}
\usepackage{enumitem}

\newtheorem{theorem}{Theorem}[section]
\newtheorem{lemma}[theorem]{Lemma}

\newtheorem{proposition}[theorem]{Proposition} 

\newtheorem{conjecture}[theorem]{Conjecture}
\theoremstyle{definition}
\newtheorem{definition}[theorem]{Definition}
\theoremstyle{remark}

\advance \topmargin by -\headheight
\advance \topmargin by -\headsep
\evensidemargin \oddsidemargin
\newcommand{\Int}{\mathrm{Int}}
\title{Off-diagonal ordered Ramsey numbers of matchings}
\author{Dhruv Rohatgi \\ MIT \\ drohatgi@mit.edu}
\date{}

\begin{document}
\maketitle

\begin{abstract}
For ordered graphs $G$ and $H$, the ordered Ramsey number $r_<(G,H)$ is the smallest $n$ such that every red/blue edge coloring of the complete graph on vertices $\{1,\dots,n\}$ contains either a blue copy of $G$ or a red copy of $H$, where the embedding must preserve the relative order of vertices. One number of interest, first studied by Conlon, Fox, Lee, and Sudakov, is the ``off-diagonal'' ordered Ramsey number $r_<(M, K_3)$, where $M$ is an ordered matching on $n$ vertices. In particular, Conlon et al.~asked what asymptotic bounds (in $n$) can be obtained for $\max r_<(M, K_3)$, where the maximum is over all ordered matchings $M$ on $n$ vertices. The best-known upper bound is $O(n^2/\log n)$, whereas the best-known lower bound is $\Omega((n/\log n)^{4/3})$, and Conlon et al.~hypothesize that $r_<(M, K_3) = O(n^{2-\epsilon})$ for every ordered matching $M$. We resolve two special cases of this conjecture. We show that the off-diagonal ordered Ramsey numbers for matchings in which edges do not cross are nearly linear. We also prove a truly sub-quadratic upper bound for random matchings with interval chromatic number $2$.
\end{abstract}

\section{Introduction}

A classical area of extremal combinatorics is Ramsey theory. Introduced by Ramsey \cite{Ramsey1930} and popularized by Erd\H{o}s and Szekeres \cite{Erdos1935}, the Ramsey number of a graph $G$, commonly denoted by $r(G)$, is the smallest $n$ so that every edge bicoloring of the complete graph $K_n$ contains a monochromatic copy of $G$. Shrinking the sizable gap between the asymptotic upper/lower bounds on $r(K_n)$ has been a major open problem for decades, spurring extensive work on a plethora of related questions in Ramsey theory.

One variant of Ramsey numbers which has recently received attention is the analogue for ordered graphs. An \textit{ordered graph} on $[n]$ is a graph on $n$ vertices which are given distinct labels in $\{1,\dots,n\}$. Given an ordered graph $G$, the \textit{ordered Ramsey number} of $G$, denoted by $r_<(G)$, is the smallest $n$ so that every edge bicoloring of the ordered complete graph on $n$ vertices contains a monochromatic copy of $G$ which preserves the relative vertex ordering of $G$. As with the unordered case, one can define the \textit{off-diagonal} ordered Ramsey number of two graphs $G$ and $H$, denoted by $r_<(G,H)$, as the smallest $n$ so that every edge bicoloring of the ordered complete graph on $n$ vertices contains either an order preserving red copy of $G$ or an order preserving blue copy of $H$.

The first systematic studies of ordered Ramsey numbers were conducted by Conlon, Fox, Lee, and Sudakov \cite{Conlon2014} and by Balko, Cibulka, Kr\'al, and Kyn\v cl \cite{Balko2015}. However, as pointed out by the authors of \cite{Conlon2014}, a number of classic results in extremal combinatorics can be reinterpreted as statements about ordered Ramsey numbers. For instance, Erd\H{o}s and Szekeres proved \cite{Erdos1935} that every sequence of at least $(n-1)^2 + 1$ distinct numbers contains either an increasing subsequence of length $n$ or a decreasing subsequence of length $n$. This result is implied by the bound $r_<(P_n, K_n) \leq (n-1)^2 + 1$, where $P_n$ is the $n$-vertex path imbued with the natural monotonic ordering: for any sequence of $n$ distinct numbers $x_1,\dots,x_n$, color $(i,j)$ red if $x_i < x_j$ and blue otherwise.

Perhaps the simplest nontrivial family of ordered graphs from the perspective of ordered Ramsey theory is \emph{matchings}, in which every vertex has degree $1$. Conlon, Fox, Lee, and Sudakov provide a number of bounds for general matchings, for matchings satisfying certain properties, and for off-diagonal ordered Ramsey numbers involving matchings. Relevant to this paper is their work on bounding the largest possible value of $r_<(M, K_3)$, where $M$ is a matching. They have the following result:

\begin{theorem}[Conlon, Fox, Lee, and Sudakov \cite{Conlon2014}]
There are positive constants $c_1$ and $c_2$ such that for all even positive integers $n$,
$$c_1 \left ( \frac{n}{\log n} \right)^{4/3} \leq \max_{M} \, r_<(M, K_3) \leq c_2 \frac{n^2}{\log n}$$ where the maximum is taken over all ordered matchings $M$ on $n$ vertices.
\end{theorem}

The upper bound in this theorem is in some sense trivial. Since every graph on $n$ vertices embeds in the complete graph $K_n$, and the ordered Ramsey number $r_<(K_n, K_3)$ is equal to the Ramsey number $r(n, 3)$, which has been asymptotically determined \cite{Ajtai1980, Kim1995} to be $\Theta(n^2/\log n)$, it follows (as pointed out in \cite{Conlon2014}) that $r_<(M, K_3) = O(n^2/\log n)$ for a matching $M$ on $n$ vertices. However, this bound does not make use of any properties of matching graphs, only making use of the fact that every graph on $n$ vertices can be embedded in $K_n$. For this reason and perhaps other reasons, Conlon, Fox, Lee, and Sudakov hypothesize \cite{Conlon2014} that the upper bound can be improved to $r_<(M, K_3) \leq n^{2-\epsilon}$ for some $\epsilon > 0$.

We contribute two results in the direction of this conjecture. We first look at the special case of ordered matchings where the edges do not cross. That is, for any two edges $(i,j)$ and $(k,l)$ with $i<j$ and $k<l$, the intervals $[i,j]$ and $[k,l]$ are either disjoint or nested one inside the other. We call the matchings which satisfy this condition ``parenthesis matchings'', after the useful fact that these matchings correspond with balanced parenthesis sequences. Indeed, it is this correspondence which partially motivates our proof of the following theorem.

\theoremstyle{theorem}
\newtheorem*{thm:pmatchings}{Theorem \ref{thm:pmatchings}}
\begin{thm:pmatchings}
For any $\epsilon > 0$ there is a constant $c$ such that every parenthesis matching $M$ on $n$ vertices has $$r_<(M, K_3) \leq cn^{1+\epsilon}.$$
\end{thm:pmatchings}

To state our second result, we must define the \textit{interval chromatic number} of an ordered graph. Analogous to the chromatic number of an unordered graph, the interval chromatic number $\chi_<(G)$ of a graph $G$ is the minimum number of contiguous intervals into which the vertex set must be split so that each interval is an independent set in $G$.

Conlon, Fox, Lee, and Sudakov present a number of general results accompanied by much stronger specific results for matchings with small interval chromatic number \cite{Conlon2014}. In a similar spirit, we prove a sub-quadratic bound on $r_<(M, K_3)$ for random matchings with interval chromatic number $2$.

\theoremstyle{theorem}
\newtheorem*{thm:rmatchings}{Theorem \ref{thm:rmatchings}}
\begin{thm:rmatchings}
There is a constant $c$ such that for every even $n$, if an ordered matching $M$ on $n$ vertices with interval chromatic number $2$ is picked uniformly at random, then $$r_<(M, K_3) \leq cn^{\frac{24}{13}}$$ with high probability.
\end{thm:rmatchings}
 
Observe that the statement is not probabilistic over bicolorings; rather, it is a true Ramsey-type result which applies to almost all matchings.

\subsection{Roadmap}

We outline the remainder of this paper. In Section~\ref{subsection:parenthesis}, we achieve a nearly linear bound for matchings whose edges do not cross. In Section~\ref{subsection:random}, we obtain a slightly sub-quadratic bound for random matchings with interval chromatic number $2$. Finally, in Section~\ref{section:conclusion} we outline possible directions for future research, describing a few of the many interesting questions about ordered Ramsey numbers which remain open.

Throughout the paper, we make no serious attempts to optimize constants.

\section{Parenthesis Matchings}\label{subsection:parenthesis}

Earlier we defined ``parenthesis matchings'' as matchings for which the edges do not cross. We claim without proof that every parenthesis matching corresponds uniquely with a balanced parenthesis sequence---that is, a sequence of correctly matched open and close parentheses. The bijection is straightforward; each matched pair of parentheses corresponds with an edge in the matching. See Figure~\ref{figure:pmatching} for an example.

\begin{figure}
\centering
\includegraphics[width=0.5\textwidth]{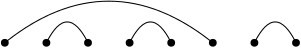}
\caption{The parenthesis matching corresponding to the parenthesis sequence $(()())()$.}\label{figure:pmatching}
\end{figure}

We start with perhaps the simplest nontrivial parenthesis matching, and work our way up to general parenthesis matchings. Define the nested matching graph $NM_k$ of size $k$ to be the graph on $[2k]$ where $(i,j)$ is an edge if and only if $i + j = 2k+1$. We establish the off-diagonal ordered Ramsey number of $NM_k$ up to constant factors:

\begin{proposition}\label{prop:nest}
For any positive integer $k$, $$4k - 2 < r_<(NM_k, K_3) \leq 6k.$$
\end{proposition}

\begin{proof}
The lower bound follows from a simple construction: color the ordered complete graph $K_{4k-2}$ such that $\{1,\dots,2k-1\}$ and $\{2k, \dots, 4k-2\}$ form two red cliques, and all remaining edges are blue. Then there are no blue triangles, and no red edge $(i,j)$ has $|i-j| > 2k-2$, so there cannot be a red matching on $2k$ vertices.

For the upper bound, pick an arbitrary bicoloring of $K_{6k}$. Suppose the graph contains no blue copies of $K_3$. If any vertex has blue degree at least $2k$, then there is a red clique of size $2k$, which must contain $M$. Otherwise, the number of blue edges is at most $6k^2$. Hence, the number of red edges is at least $12k^2 - 3k$. Let $E_R$ be the set of red edges, and define a strict partial order on $E_R$ by $(i,j) < (l,m)$ if $l < i < j < m$. We wish to show that there is a ``chain'' of edges $e_1,\dots,e_k$ with $e_1 < \dots < e_k$. 

For the sake of contradiction, suppose the contrary, so every chain has length at most $k-1$. Define a function $L\colon E_R \to \{1, \dots, k-1\}$ where $L(e)$ is the longest chain ending at $e$. Observe that $L^{-1}(n)$ is an ``anti-chain'' for each $n \in [k-1]$. That is, for any $e_1, e_2 \in L^{-1}(n)$, we cannot have $e_1 < e_2$ nor $e_2 < e_1$. 

Applying the pigeonhole principle, fix some $n$ such that $|L^{-1}(n)| \geq 12k$. For $1 \leq i \leq 6k$ let $a_i$ be the minimum index $j$ such that $(i, j) \in |L^{-1}(n)|$, and let $b_i$ be the maximum such index. Then $\sum_{i=1}^{6k} (b_i + 1 - a_i) \geq 12k$, so $\sum_{i=1}^{6k} (b_i - a_i) \geq 6k$. It follows that there exist indices $i < j$ with $b_i > a_j$. But then $i < j < a_j < b_j$, so edges $(i, a_i)$ and $(j, b_j)$ are comparable. This contradicts our claim that $L^{-1}(n)$ is an anti-chain, so there must be a chain of length at least $k$. The edges in the chain comprise the red embedding of $NM_k$ into the graph.
\end{proof}

We believe that the upper bound is far from optimal. In particular, we make the following conjecture.

\begin{conjecture}
For any positive integer $k$, $$r_<(NM_k, K_3) = 4k - 1.$$
\end{conjecture}

The nested matching can be used to bound the corresponding ordered Ramsey numbers for a more general class of matchings. As we will build up more complex parenthesis matchings from simpler ones, we need a way to keep track of the growth of $r_<(M, K_3)$. One approach is the following lemma:

\begin{lemma}\label{lemma:nestsimple}
Let $A_1,\dots,A_{2k-1}$ be (possibly empty) balanced parenthesis sequences inducing matchings $M_1,\dots,M_{2k-1}$. Then $$(A_1(A_2(\cdots(A_{k-1}(A_k)A_{k+1})\cdots)A_{2k-2})A_{2k-1})$$ is a balanced parenthesis sequence which induces some matching $M$, with $$r_<(M, K_3) \leq r_<(NM_{k+t}, K_3),$$ where $t = \sum_{i=1}^k \max(r_<(M_i, K_3), r_<(M_{2k-i}, K_3))$.
\end{lemma}

\begin{proof}
Pick an arbitrary bicoloring of the complete graph on $r_<(NM_{k+t}, K_3)$ vertices. Assume that there is no blue copy of $K_3$. Then there is a red copy of $NM_{k+t}$. Starting with the innermost edge of the matching and working outwards, delete as many matched pairs as necessary until there is space for a red copy of $M_k$. Every deletion increases the number of inner vertices by at least one, so there will be space after at most $r_<(M_k, K_3)$ steps. Save the current innermost matched pair (which will correspond to the parentheses around $A_k$), and continue deleting subsequent matches until there is space for a red copy of $M_{k-1}$ (to the left of the saved match) and a red copy of $M_{k+1}$ (to the right of the saved match). The number of deletions is at most $\max(r_<(M_{k-1}, K_3), r_<(M_{k+1}, K_3))$; save the new innermost match.

Repeating the above process $k-2$ more times yields a complete red copy of $M$. Note that the process does not run out of matches, since only $k$ matches are saved, and at most $t$ matches are deleted.
\end{proof}

In the above lemma, the Ramsey number of each matching $M_i$ is multiplied by a constant factor arising from the Ramsey number of a nested matching $NM_n$. It is possible to decrease the dependence on the central matching $M_k$, in exchange for larger constants on the remaining matchings and on the length of the matching.

\begin{lemma}\label{lemma:nestcomplex}
Let $A_1,\dots,A_{2k-1}$ be balanced parenthesis sequences inducing matchings $M_1,\dots,M_{2k-1}$. Let $M$ be the parenthesis matching induced by the expression $$(A_1(A_2(\cdots(A_{k-1}(A_k)A_{k+1})\cdots)A_{2k-2})A_{2k-1}).$$ If $l = \sum_{i \neq k} r_<(M_i, K_3)$ and $t = r_<(M_k, K_3)$, then $$r_<(M, K_3) \leq t + 20(k + l + |M_k|).$$
\end{lemma}

\begin{proof}
Pick an arbitrary bicoloring of the ordered complete graph on $t + 20(k + l + |M_k|)$ vertices. Assume that there is no blue copy of $K_3$. Let $X$ denote the first $10(k + l + |M_k|)$ vertices; let $Y$ denote the next $t$ vertices; and let $Z$ denote the remaining $10(k + l + |M_k|)$ vertices. Observe that $Y$ contains a red copy of $M_k$.

Suppose that there is a red copy of $NM_{k+l}$ in $X \cup Z$, where the first $k+l$ vertices are in $X$ and the remaining $k+l$ vertices are in $Z$. Then, just as in Lemma~\ref{lemma:nestsimple}, we can start with the innermost matching and work outwards, deleting matchings to make space for red copies of $M_1,\dots,M_{k-1}$ and $M_{k+1},\dots,M_{2k-1}$. Only $l$ matchings need be deleted, and by the end, the graph $X \cup Y \cup Z$ contains a red copy of $M$.

Now suppose the converse, so the maximum number of nested matchings from $X$ to $Z$ is less than $k + l$. As in Proposition~\ref{prop:nest}, define the natural strict partial order on the red edges between $X$ and $Z$. A set of nested edges forms a ``chain'', and the largest anti-chain contains no more than $|X| + |Z| = 20(k + l + |M_k|)$ red edges. We know that the red edges can be partitioned into less than $k+l$ anti-chains, so the number of red edges between $X$ and $Z$ is at most $20(k+l+|M_k|)(k+l)$, which we upper bound by $20(k+l+|M_k|)^2$.

Thus, the number of blue edges between $X$ and $Z$ is at least $80(k+l+|M_k|)^2$. Hence there must be a vertex $v \in X$ with at least $8(k+l+|M_k|)$ blue edges into $Z$. Since the graph was assumed to be blue $K_3$-free, it follows that the set of blue neighbors of $v$ forms a red clique of size $8(k+l+|M_k|)$. As $|M| \leq 8(k+l+|M_k|)$, we conclude that the bicoloring contains a red copy of $M$.
\end{proof}

Every parenthesis matching is in a bijection with an ordered, rooted tree. The above lemma allows us to bound the off-diagonal Ramsey number of the tree by the Ramsey numbers of all the branches off any path. Intuitively (and we will formalize the intuition later), this bound is strong on unbalanced trees and weak on well-balanced trees. For the latter case, we have the following simple lemma. While it is a special case of the above lemma aside from unimportant constant factors, we will use it for a different purpose (namely, well-balanced trees), so we state it separately for clarity.

\begin{lemma}\label{lemma:surround}
Let $A$ be a balanced parenthesis sequence inducing the matching $M$. Then $(A)$ is a balanced parenthesis sequence inducing some matching $M'$, and $$r_<(M', K_3) \leq r_<(M, K_3) + |M'| + 1.$$
\end{lemma}

\begin{proof}
Let $t = r_<(M, K_3)$ and let $n$ be the number of vertices in matching $M'$. Pick an arbitrary bicoloring of the ordered complete graph on $[t+n+1]$. Suppose there are no blue triangles. Then there is a red copy of $M$ in $\{2, \dots, t+1\}$. So if there is a red edge from $1$ to any of $\{t+2,\dots,t+n+1\}$, we have found a red copy of $M'$. Otherwise, every edge from $1$ to $\{t+2,\dots,t+n+1\}$ is blue, so $\{t+2,\dots,t+n+1\}$ form a red clique of size $n$, which must contain the matching $M'$. 
\end{proof}

With the above lemmas, we can prove a subquadratic bound on the Ramsey numbers of all balanced parenthesis matchings. Two convexity results are needed; we postpone their proofs to Appendix~\ref{section:inequalities}.

\begin{lemma}\label{lemma:convex1}
Let $a_0, a_1, a_2, \dots,a_k \geq 0$ and $\delta > 1$ and $m > 0$ be real numbers. Let $r = {m}^{-1/(\delta-1)}$. If $s = \sum_{i=0}^k a_i \geq 1$ and $a_i \leq rs$ for all $1 \leq i \leq k$, then $$m(a_0 + ca_1^\delta + \dots + ca_k^\delta) \leq cs^\delta$$ for any $c \geq m$.
\end{lemma}

\begin{lemma}\label{lemma:convex2}
Let $a_1, \dots a_k \geq 0$ and $\delta \geq 1$ be real numbers. Let $r \in (0,1)$. If $s = \sum_{i=1}^k a_i$ and $a_i \leq rs$ for all $1 \leq i \leq k$, then $$a_1^\delta + \dots + a_k^\delta \leq r^{\delta-1}s^\delta.$$
\end{lemma}

In the following proof we'll use the bijection between parenthesis matchings on $n$ vertices and ordered rooted trees of size $s = n/2 + 1$. The basic idea is to induct on tree size and decompose the tree into smaller trees by one of two methods, depending on the relative weights of the root's child subtrees.

Call an edge \textit{$r$-heavy} if $s_\text{child} \geq r \cdot s_\text{parent}$, where $s_\text{child}$ is the size of the child subtree and $s_\text{parent}$ is the size of the parent subtree. If the inequality does not hold, call the edge \textit{$r$-light}. Similarly call a vertex $r$-heavy or $r$-light if its parent edge is $r$-heavy or $r$-light, respectively.

If all children of the root are $r$-light for an appropriate choice of $r$ (slightly less than $1$), we apply the inductive hypothesis to each child separately, and use Lemma~\ref{lemma:surround} to obtain a bound for the entire tree. Since every child subtree is a constant factor smaller than the entire tree, the lemma intuitively yields a sufficiently good recurrence.

If however the root has an $r$-heavy child, Lemma~\ref{lemma:surround} does not suffice. Instead we trace a path of heavy edges from the root down, decomposing the tree into a number of branches, as well as possibly some subtrees at the tail end of the path. Here we use Lemma~\ref{lemma:nestcomplex}. We know that every branch is $(1-r)$-light, so can afford to multiply the sum of Ramsey numbers of the branches by $20$ in the lemma. We only know the tail subtrees to be $r$-light, which is why they are treated differently in the lemma.

Formalizing the above proof sketch requires some manipulation of inequalities and applications of Lemma~\ref{lemma:convex1} and Lemma~\ref{lemma:convex2}. We work through these below.

\begin{theorem}\label{thm:pmatchings}
For any $\epsilon > 0$ there is a constant $c$ such that every parenthesis matching $M$ on $n$ vertices has $r_<(M, K_3) \leq cn^{1+\epsilon}$.
\end{theorem}

\begin{proof}
Let $\epsilon > 0$. Set $r = 1 - 23^{-2/\epsilon}$, and set $c = 23/(1 - r^\epsilon)$. A parenthesis matching on $n$ vertices uniquely corresponds with an ordered rooted tree of size $s = n/2 + 1$. We induct on the tree size $s$. If $s = 1$, the corresponding matching is the empty matching on $0$ vertices, for which the claim is trivially true. Fix an ordered rooted tree of size $s > 1$, corresponding to a matching $M$. There are two cases which we will treat separately; either the tree root has an $r$-heavy child, or not.

Suppose that the tree root does not have an $r$-heavy child. Let $s_1,\dots,s_k$ be the sizes of the child subtrees of the root. Let $M_1,\dots,M_k$ be the matchings corresponding to the respective subtrees, and let $t_i = r_<(M_i, K_3)$ for each $i \in [k]$. With a slight abuse of notation, identifying the matchings with their parenthesis sequences, we have $$M = (M_1)(M_2)\dots(M_k).$$ Lemma~\ref{lemma:surround} provides the bound $r_<((M_i), K_3) \leq t_i + 2s_i + 1$. Since the Ramsey number of a union of ordered graphs on disjoint intervals of vertices is subadditive, it follows that $$r_<(M, K_3) \leq \sum_{i=1}^k (t_i + 2s_i + 1) \leq 3s + \sum_{i=1}^k t_i.$$ By the inductive hypothesis and Lemma~\ref{lemma:convex2} (using the assumption that every subtree is $r$-light), we have $$r_<(M,K_3) \leq 3s + \sum_{i=1}^k cs_i^{1+\epsilon} \leq 3s + cr^\epsilon s^{1+\epsilon} \leq cs^{1+\epsilon}.$$ The last step follows since $c$ was chosen to be sufficiently large.

The remaining case to consider is if the tree root has a heavy child. Then there is some path which starts at the root and consists entirely of heavy edges (possibly only one edge, or possibly more). Let $s^b_1,\dots,s^b_k$ be the sizes of all subtrees which branch off the heavy path, and let $s^h$ be the (vertex) size of the heavy path. Let $M^b_1,\dots,M^b_k$ be the corresponding matchings, and let $t^b_i = r_<(M^b_i, K_3)$ for each $i \in [k]$. For ease of notation, suppose that the deepest vertex in the heavy path has $k'$ children, and its child subtrees are indexed $1 \dots k'$. The whole matching $M$ can be decomposed into a nested matching along with embedded matchings $(M^b_1),\dots,(M^b_k)$. For instance, if $k = 3$ and $k' = 1$ then one possibility is $M = ((M^b_2)(()(M^b_1))(M^b_3))$. By Lemma~\ref{lemma:surround}, the following bound holds for every matching $M^b_i$: $$r_<((M^b_i), K_3) \leq t^b_i + 3s^b_i.$$ So by Lemma~\ref{lemma:nestcomplex}, we have $$r_<(M, K_3) \leq \sum_{i=1}^{k'} \left(t^b_i + 3s^b_i\right) + 20 \left(s^h + \sum_{i=k'+1}^k \left(t^b_i + 3s^b_i\right) + \sum_{i=1}^{k'} s^b_i \right).$$ By the inductive hypothesis, it follows that 
\begin{eqnarray*}
r_<(M, K_3) 
&\leq& \sum_{i=1}^{k'} \left(c\left(s^b_i\right)^{1+\epsilon} + 3s^b_i \right) \\
& & {} + 20 \left ( s^h + \sum_{i=k'+1}^k \left(c\left(s^b_i\right)^{1+\epsilon} + 3s^b_i\right) + \sum_{i=1}^{k'} s^b_i \right).
\end{eqnarray*}

Reordering terms and absorbing the term $3 \sum_{i=k'+1}^k s^b_i$ into the outer constant factor through the bound $c \geq 20$, we get
\begin{eqnarray}\label{eq:overall}
r_<(M, K_3)
&\leq& c\sum_{i=1}^{k'} \left(s^b_i\right)^{1+\epsilon} + 23 \sum_{i=1}^{k'} s^b_i \nonumber \\
& & {} + 23 \left ( s^h + c \sum_{i=k'+1}^k \left(s^b_i\right)^{1+\epsilon} \right). 
\end{eqnarray}

To bound the first two terms of Equation~\ref{eq:overall}, we observe that for each $i \leq k'$, subtree $i$ is $r$-light, and therefore $s^b_i \leq rs$. An application of Lemma~\ref{lemma:convex2}, along with the bound $cr^\epsilon + 23 \leq c$, gives 
\begin{align}\label{eq:deepsubtrees}
c\sum_{i=1}^{k'} \left(s^b_i\right)^{1+\epsilon} + 23 \sum_{i=1}^{k'} s^b_i 
&\leq cr^\epsilon \left ( \sum_{i=1}^{k'} s^b_i \right)^{1+\epsilon} + 23 \sum_{i=1}^{k'} s^b_i \nonumber \\
& \leq c \left ( \sum_{i=1}^{k'} s^b_i \right)^{1+\epsilon}.
\end{align} 
For the remaining terms of Equation~\ref{eq:overall}, observe that for any $i > k'$, subtree $i$ has an $r$-heavy sibling, so $s^b_i$ is at most $1-r$ times the parent's subtree size, and therefore at most $(1-r)s$. We will use one of two approaches (below, \textbf{A} and \textbf{B}) depending on the cumulative weight of these subtrees.
\begin{itemize}[leftmargin=*]
\item[\textbf{A.}] If $s^h + \sum_{j=k'+1}^{k} s^b_j \geq 23^{-1/\epsilon}s$, then we can bound $s^b_i \leq 23^{1/\epsilon}(1-r)\left(s^h + \sum_{j=k'}^{k} s^b_j\right)$ for all $i > k'$. We know that $c \geq 23$ and $23^{1/\epsilon}(1-r) \leq 23^{-1/\epsilon}$, so an application of Lemma~\ref{lemma:convex1} yields 
\begin{equation}\label{eq:branch:heavy}
23 \left ( s^h + c\sum_{i = k'+1}^k \left(s^b_i\right)^{1+\epsilon} \right) \leq c \left ( s^h + \sum_{i=k'+1}^k s^b_i \right)^{1+\epsilon}.
\end{equation}

Summing together the bounds from Equation~\ref{eq:deepsubtrees} and Equation~\ref{eq:branch:heavy} and applying the most basic convexity bound, we get the desired bound
\begin{align*}
r_<(M, K_3) 
&\leq c \left ( \sum_{i=1}^{k'} s^b_i \right)^{1+\epsilon} + c \left ( s^h + \sum_{i=k'+1}^k s^b_i \right)^{1+\epsilon} \\
&\leq cs^{1+\epsilon}.
\end{align*}

\item[\textbf{B.}] If $s^h + \sum_{j=k'+1}^k s^b_j < 23^{-1/\epsilon} s$, then we are unable to bound $s^b_i$ against $s^h + \sum_{j=k'+1}^k s^b_j$, but we know that the latter quantity is much smaller than $s$. So we instead use the weak bound
\begin{equation}\label{eq:branch:light}
23\left(s^h + c\sum_{i=k'+1}^k \left(s^b_i\right)^{1+\epsilon}\right) \leq 23c\left(s^h + \sum_{i=k'+1}^k s^b_i \right)^{1+\epsilon}.
\end{equation}

Now we combine Equation~\ref{eq:deepsubtrees} with Equation~\ref{eq:branch:light}, using the simple inequality $(1-x)^{1+\epsilon} + 23x^{1+\epsilon} \leq 1$ for $x \in (0, 23^{-1/\epsilon})$, and obtain
\begin{align*}
r_<(M, K_3) 
&\leq c \left ( \sum_{i=1}^{k'} s^b_i \right)^{1+\epsilon} + 23c \left ( s^h + \sum_{i=k'+1}^k s^b_i \right)^{1+\epsilon} \\
&\leq cs^{1+\epsilon}.
\end{align*}

\end{itemize}

This completes the induction.
\end{proof}

\section{Random Matchings with $\chi_<(M) = 2$}\label{subsection:random}

Recall that the \textit{interval chromatic number} $\chi_<(G)$ of an ordered graph $G$ is the minimum number of contiguous intervals into which the vertex set must be split so that each interval is an independent set in $G$.

In this section, we show that for almost every matching $M$ with interval chromatic number $2$, the bound of $\widetilde{O}(n^{2})$ on $r_<(M, K_3)$ can be beaten. More specifically, we exhibit a condition on $M$ which is sufficient to guarantee an improved bound on $r_<(M, K_3)$, and then prove that a random matching with interval chromatic number $2$ satisfies this condition with high probability.

\begin{figure}
\centering
\includegraphics[width=0.5\textwidth]{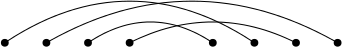}
\caption{A matching $M$ with interval chromatic number $2$, and corresponding permutation $\pi(M) = (2,4,1,3)$.}\label{figure:icn2}
\end{figure}

The set of matchings on $2n$ vertices with interval chromatic number $2$ is in bijection with the permutation group $S_n$, and it is often notationally convenient to examine the permutation corresponding to a given matching. See Figure~\ref{figure:icn2} for an example.

\begin{definition}
Let $M$ be an ordered matching on $[2n]$ with interval chromatic number $2$. Then its ``corresponding permutation'' $\pi(M)$ is the permutation on $[n]$ which maps $i$ to $j-n$ for every edge $(i,j) \in M$.
\end{definition}

\begin{definition} We say that a given permutation $\pi \in S_n$ contains an ``exact pattern'' $\rho$ if $\rho$ is an ordered subset of $[n]$ and there are indices $1 \leq i_1 < \dots < i_k \leq n$, where $k = |\rho|$, such that $\pi(i_j) = \rho(j)$ for all $j$.
\end{definition}

For instance, the permutation $\pi = (3,5,6,1,2,4)$ contains the exact pattern $(6,1,4)$ but does not contain the exact pattern $(1,2,3)$.

We are interested in using exact patterns as a metric for the ``intersection'' of two permutations. Specifically, we make the following definition.

\begin{definition}
Let $\pi, \sigma \in S_n$ be permutations. Define the ``ordered intersection'' of $\pi$ and $\sigma$, denoted $\Int(\pi, \sigma)$, to be the largest $k$ such that both $\pi$ and $\sigma$ share an exact pattern of length $k$.
\end{definition}

In the theorem below, we do something slightly stronger than bounding the Ramsey number $r_<(M, K_3)$ for certain matchings $M$. Rather, we show that in a blue $K_3$-free graph on $2n$ vertices, there is a tradeoff between finding a red copy of the matching $M$ in the bipartite subgraph $[1,n] \cup [n+1, 2n]$ and finding a large red clique (which of course contains every matching of that size) in $[1,n]$ or symmetrically in $[n+1, 2n]$.

Observe that a trivial claim, following immediately from unordered Ramsey theory, is ``every bicoloring of $2n$ vertices contains either a blue triangle of a red clique of size $\Theta((n \log n)^{1/2})$.'' This is of course the best possible claim, in that $r(n, 3) = \Theta(n^2/\log n)$. The following theorem shows that the claim can be improved---that is, there is a red clique of size $\omega((n \log n)^{1/2})$---under an added assumption about the absence of a red matching satisfying certain conditions.

\begin{theorem}\label{theorem:cliquematching}
Fix $\epsilon \in (0, 1)$ and $\alpha, \beta > 0$ with $\alpha + \beta \leq \epsilon/4$. Let $M$ be an ordered matching on $2n^{1/2 + \alpha}$ vertices with interval chromatic number $2$, such that the corresponding permutation $\pi = \pi(M)$ satisfies $\Int(\pi(M), \pi(M)+h) \leq n^{(1-\epsilon)(1/2+\alpha)}$ for every $h \in [n^{1/2 + \alpha}]$. Then every red/blue coloring of the ordered complete graph on $[2n]$ contains either:
\begin{itemize}
\item a blue copy of $K_3$,
\item a red copy of $K_{n^{1/2 + \beta}/4 - n^{\epsilon/4}/2}$, or
\item a red copy of $M$ within the bipartite subgraph $[1,n] \cup [n+1,2n]$.
\end{itemize}
\end{theorem}

\begin{proof}
Fix a bicoloring $C$ of the ordered complete graph on $[2n]$, and suppose that it contains none of the hypothesized blue or red structures. Then in particular, for $1 \leq i \leq n+1-n^{1/2+\alpha}$, we know that there are no red copies of $M$ between $[1,n]$ and $[n+i, n+i+n^{1/2+\alpha})$.

Fix some $i \leq n^{\epsilon/4}$. Let $v_1(i)$ be the first vertex in $[n]$ such that $C(v_1(i), i+\pi(1))$ is red (or $v_1(i) = \infty$ if no such vertex exists). Let $b_2(i)$ be the first vertex in $[n]$ after $v_1(i)$ such that $C(v_2(i), i+\pi(2))$ is red (or, again, $v_2(i) = \infty$ if no such vertex exists). Iteratively define $v_3(i), \dots, v_{n^{1/2+\alpha}}(i)$ in the same way. Also let $f(i)$ be the first index at which $v_{f(i)}(i) = \infty$. By our assumption that $[1,n] \cup [n+i,n+i+n^{1/2+\alpha})$ is red $M$-free, this index exists.

\begin{figure}
\centering
\includegraphics[width=0.5\textwidth]{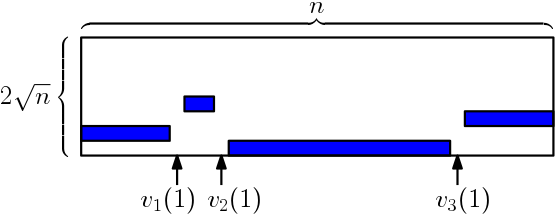}
\caption{One possibility for the set of segments $F(1)$ in the blue adjacency matrix, if $\pi(m) = (2, 4, 1, 3)$.} \label{figure:segments}
\end{figure}

The vertices $v_1(i),\dots,v_{f(i)-1}(i)$ demarcate $f(i)$ blue segments in the adjacency matrix of $[1,n] \cup [n+1,n+2n^{1/2+\alpha}]$. That is, for $1 \leq j \leq f(i)$ we have $C(k, i+\pi(j))$ is blue for all $v_{j-1}(i) < k < v_j(i)$ (where for convenience we set $v_0(i) = 0$ and $v_{f(i)}(i) = n+1$). Treating $C$ as an $n \times 2n^{1/2+\alpha}$ matrix, each segment is in a distinct row, and the segments occupy distinct intervals of columns, covering a total of at least $n - n^{1/2+\alpha}$ columns. If any segment had length at least $n^{1/2+\beta}$, then some vertex would have $n^{1/2+\beta}$ blue edges, so the coloring would contain either a blue triangle or a red $K_{n^{1/2+\beta}}$. So henceforth we assume that every segment has length at most $n^{1/2+\beta}$.

For each $i \leq n^{\epsilon/4}$ let $F(i)$ be the set of $f(i)$ blue segments as defined above (see Figure~\ref{figure:segments} for an example). We seek to lower bound the number of blue edges in $F(i)$ which are not contained in any $F(i')$ for $i' < i$. So fix $i' < i \leq n^{\epsilon/4}$. Suppose that there are $k$ segments in $F(i)$ which intersect with segments in $F(i')$. 

Since each segment in $F(i)$ is in a different row, as is each segment of $F(i')$, each of the $k$ intersecting segments in $F(i)$ intersects with a unique segment in $F(i')$. Suppose that $s_1, s_2 \in F(i)$ and $t_1, t_2 \in F(i')$ where $s_1$ intersects $t_1$ and $s_2$ intersects $t_2$. Then $\text{row}(s_1) = \text{row}(t_1)$, and $\text{row}(s_2) = \text{row}(t_2)$. And since the segments $F(i)$ hit disjoint intervals of columns, as do the segments $F(i')$, we have $\text{columns}(s_1)$ is ``left'' of $\text{columns}(s_2)$ in the adjacency matrix if and only if $\text{columns}(t_1)$ is ``left'' of $\text{columns}(t_2)$. So the $k$ intersecting segments define an exact pattern in both $\pi+i'$, which describes the row indices of the segments $F(i')$, and $\pi+i$, which describes the row indices of the segments $F(i)$. It follows that $k$ is at most $\Int(\pi, \pi+i-i')$, which is by assumption at most $n^{(1-\epsilon)(1/2+\alpha)}$. Summing over all $i' < i$, at most $in^{(1-\epsilon)(1/2+\alpha)}$ segments in $F(i)$ intersect with previous segments.

Every segment has length at most $n^{1/2+\beta}$ by assumption. Thus, for each $i \leq n^{\epsilon/4}$, the blue segments in $F(i)$ contribute at least $$n - n^{1/2+\alpha} - in^{(1-\epsilon)(1/2+\alpha)}n^{1/2+\beta} = n - n^{1/2+\alpha} - in^{1-\epsilon/2+\beta+(1-\epsilon)\alpha}$$ new blue edges. When $i=1$ the contribution is $n - n^{1/2+\alpha}$; when $i = n^{\epsilon/4}$, the contribution is at least $-n^{1/2+\alpha}$. The contributions decrease linearly, so in total there are at least $n^{1 + \epsilon/4}/2 - n^{1/2 + \alpha + \epsilon/4}$ blue edges in the bipartite graph $[1,n] \cup [n+1,n+2n^{\alpha+1/2}]$. So some vertex has blue degree at least $n^{1/2 + \beta}/4 - n^{\epsilon/4}/2$, implying that there is either a blue triangle or a red clique of size $n^{1/2 + \beta}/4 - n^{\epsilon/4}/2$.
\end{proof}

We seek to show that for random permutations $\pi$ and for any integer $h$, the intersection of $\pi$ with the shifted permutation $\pi + h$ is sublinear in the length of $\pi$ with high probability. The general outline of the proof is as follows. We bound the expected number of long exact patterns contained in both $\pi$ and $\pi+h$. To do so, we of course sum over all long exact patterns, splitting into two cases. If the exact pattern $\rho$ has small intersection with $\rho+h$, we can straightforwardly obtain a good bound on the probability that $\rho$ embeds into both permutations. However, if $\rho$ has large intersection with $\rho+h$, we cannot do so. Instead we show that the number of such exact patterns is extremely small.

The following lemma formalizes the last step of the above outline.

\begin{lemma}\label{lemma:largeinter}
Fix positive integers $n$, $k \leq n$, and $h$. Pick an exact pattern $\rho$ of length $k$ from $[n]$ uniformly at random. Then the probability that the set intersection $\rho \cap (\rho + h)$ has size at least $t$, and there exists some permutation $\pi \in S_n$ such that $\rho$ and $\rho+h$ are both exact patterns in $\pi$, does not exceed $$\frac{2^{2k-t}k^{k-t}}{k!}.$$ 
\end{lemma}

\begin{proof}
Observe that it is possible to pick an exact pattern uniformly at random by two independent choices: first, pick an unordered subset of $[n]$ with size $k$. Second, pick some ordering for the subset. We will show that for any unordered subset $U  \subseteq [n]$ with size $k$ such that $|U \cap (U + h)| \geq t$, if we pick an ordering on $U$ uniformly at random and thereby induce an exact pattern $\rho$, then the probability that there exists a permutation $\pi$ in which $\rho$ and $\rho+h$ are both exact patterns does not exceed $2^{2k-t}k^{k-t}/k!$. This will prove the lemma.

Fix any $U \subseteq [n]$ with $|U| = k$ and $|U \cap (U+h)| \geq t$. The number of elements $a \in U$ such that $a-h \not \in U$ does not exceed $k - t$, so $U$ can be partitioned into $k - t$ arithmetic progressions, each with common difference $h$.

Pick some permutation $\sigma \in S_k$. This yields an ordering of $U$, in which the smallest element of $U$ is placed in position $\sigma(1)$, and so forth. Hence, an exact pattern $\rho$ is induced. Suppose that the ordering is ``compatible'': that is, $\rho$ and $\rho+h$ are both exact patterns in some permutation $\pi$. Since $\rho$ and $\rho+h$ fix the order in $\pi$ of the sets of elements $U$ and $U+h$ respectively, it must hold that $U \cap (U+h)$ has the same order in $\rho$ and $\rho+h$. Pick any arithmetic progression $\{a + ih\}_{i=0}^m \subseteq U$. We have that $a+ih$ precedes $a+(i+1)h$ in $\rho$ if and only if $a+ih$ precedes $a+(i+1)h$ in $\rho+h$, or equivalently $a+(i-1)h$ precedes $a+ih$ in $\rho$. So the arithmetic progression must either have a monotone increasing order or a monotone decreasing order in $\rho$.

The key observation was that for any $a,b \in U$ where neither $a$ nor $b$ is the first term in its arithmetic progression, $a$ precedes $b$ in $\rho$ if and only if $a-h$ precedes $b-h$. We use this observation to bound the total number of compatible orderings. There are $2^{k-t}$ ways to assign a direction to each progression, either monotone increasing or monotone decreasing. Fix one such assignment, and suppose that $m_\text{inc}$ progressions are monotone increasing. There are at most $2^k$ ways to pick the subset of locations $L_\text{inc} \subseteq [k]$ to which the increasing-ordered progressions are assigned. It remains to pick an embedding of the increasing-ordered progressions in $L_\text{inc}$, and an embedding of the decreasing-ordered progressions in $[k] \setminus L_\text{inc}$. The two cases are symmetric, so we consider the increasing-ordered progressions. 

For notational convenience, arbitrarily index the increasing-ordered progressions $A_1,\dots,A_{m_\text{inc}}$. Now define a map $\Phi: S_{|L_\text{inc}|} \to L_\text{inc}^{m_\text{inc}}$ from embeddings of the increasing-ordered progressions into $L_\text{inc}$ (which are in bijection with the permutations $S_{|L_\text{inc}|}$) to tuples $(v_1,\dots,v_{m_\text{inc}})$, where $v_i$ is the index assigned to the first element of progression $A_i$. 

We claim that the restriction of $\Phi$ to compatible embeddings is injective. Pick two different compatible orderings of $U$, inducing exact patterns $\rho_1$ and $\rho_2$, and assume for the sake of contradiction that $\Phi(\rho_1) = \Phi(\rho_2)$. Suppose that $j$ is the first index at which $\rho_1$ and $\rho_2$ differ. By assumption, the first term of each arithmetic progression has the same index in $\rho_1$ and $\rho_2$. Therefore neither $\rho_1(j)$ nor $\rho_2(j)$ is a first term in its progression. Now observe that $\rho_1(j)$ precedes $\rho_2(j)$ in $\rho_1$, but $\rho_2(j)$ precedes $\rho_1(j)$ in $\rho_2$. Hence, $\rho_1(j) - h$ precedes $\rho_2(j) - h$ in $\rho_1$, and in $\rho_2$ the opposite holds. However, $\rho_1(j)-h$ and $\rho_2(j)-h$ are both in the first $j-1$ terms of $\rho_1$, which are equal to the first $j-1$ terms of $\rho_2$. So one of the relative orderings is impossible! Contradiction, so the restriction of $\Phi$ is injective.

Thus, there are at most $|L_\text{inc}|^{m_\text{inc}}$ ways to compatibly embed the increasing-ordered progressions into $L_\text{inc}$, and similarly there are at most $(k - |L_\text{inc}|)^{k-t-m_\text{inc}}$ ways to embed the decreasing-ordered progressions into $[k] \setminus L_\text{inc}$. So the total number of compatible orderings is at most $2^{k-t} 2^k k^{k-t}$. Since the total number of orderings is $k!$, the result follows.
\end{proof}

Now we can prove our desired result on random permutations.

\begin{lemma}\label{lemma:ordint}
Fix some $\alpha > 0$ and some positive integers $n$ and $h$. If $\pi \in S_n$ is a permutation chosen uniformly at random, then $$\Pr\left[\Int(\pi, \pi+h) \geq n^{2/3 + \alpha}\right] \leq \left(e^5 n^{-3\alpha/2}\right)^{n^{2/3+\alpha}}.$$
\end{lemma}

\begin{proof}
We proceed by bounding the expected value of $\Int(\pi,\pi+h)$. Let $k = n^{3/4}$. Pick any exact pattern $\rho$ of size $k$ in $[n]$. Then $\rho$ is contained in both $\pi$ and $\pi+h$, for any permutation $\pi \in S_n$, if and only if $\rho$ and $\rho-h$ are both contained in $\pi$. If the smallest element of $\rho$ is less than $h+1$, then $\rho-h$ cannot be contained in any permutation, so assume the contrary. 

The probability that $\rho$ and $\rho-h$ are both exact patterns in a random permutation $\pi \in S_n$ is at most the probability that $\rho \cap (\rho - h)$ and $\rho \setminus (\rho - h)$ and $(\rho-h) \setminus \rho$ are all exact patterns in $\pi$. Here, the intersection/difference of two exact patterns is taken to be the set-theoretic intersection/difference, ordered according to whichever exact pattern contains the set (and picking either pattern if both contain the set). But these three exact patterns are disjoint, so the corresponding events are independent. Suppose that $m(\rho) = |\rho \cap (\rho-h)|$. Since an exact pattern of length $r$ is contained in a random permutation with probability $1/r!$, we have that $\rho$ and $\rho-h$ are contained in a random $\pi \in S_n$ with probability at most $$\frac{1}{m(\rho)!} \cdot \frac{1}{(k-m(\rho))!^2}.$$

Observe that as a function of $m$, the above fraction is largest when $m(\rho) \approx k - \sqrt{k}$, and is increasing on $[1, k-\sqrt{k}]$ and decreasing on $[k-\sqrt{k}, k]$. Hence, the bound is strong for $m(\rho)$ small. Summing over all exact patterns $\rho$ with $m(\rho) \leq k/2$, and using the trivial bound that the number of exact patterns is $n!/(n-k)!$, we have that
$$\mathbb{E}\left[\text{\# contained patterns $\rho$ with $m(\rho) \leq k/2$} \right] \leq \frac{n!}{(n-k)!(k/2)!^3}.$$ The expectation is taken over permutations $\pi \in S_n$, and a ``contained pattern'' is an exact pattern $\rho$ such that $\rho$ and $\rho-h$ are contained in $\pi$.

To bound the expectation for patterns $\rho$ with $m(\rho) > k/2$, we first discard the patterns $\rho$ for which there is no permutation $\pi$ containing both $\rho$ and $\rho-h$. Now Lemma~\ref{lemma:largeinter} gives that the number of remaining patterns is only $$\frac{2^k k^{k/2}}{k!} \frac{n!}{(n-k)!} = \binom{n}{k} 2^k k^{k/2}.$$ Using this result and assuming the worst case that $m(\rho) = k-\sqrt{k}$, we get
$$\mathbb{E}\left[\text{\# contained patterns $\rho$ with $m(\rho) > k/2$} \right] \leq \binom{n}{k} \frac{2^kk^{k/2}}{(k-\sqrt{k})!(\sqrt{k})!^2}.$$
Putting everything together, simplifying, and substituting $k = n^{2/3+\alpha}$,
\begin{align*}
\mathbb{E}\left[\text{\# contained patterns}\right]
&\leq \frac{n!}{(n-k)!(k/2)!^3} + \binom{n}{k} \frac{2^kk^{k/2}}{(k-\sqrt{k})!} \\
&\leq \frac{n^k(2e)^{3k/2}}{k^{3k/2}} + \frac{n^k2^ke^{2k}}{k^{k/2}(k-\sqrt{k})^{k-\sqrt{k}}k^{\sqrt{k}}} \\
&\leq \frac{n^k(2e)^{3k/2}}{k^{3k/2}} + \frac{n^k2^{k+2\sqrt{k}}e^{2k}}{k^{3k/2}} \\
&\leq \left (e^5 n^{-3\alpha/2}\right)^{n^{2/3+\alpha}}. \qedhere
\end{align*}
\end{proof}

The above lemma and Theorem~\ref{theorem:cliquematching} imply the main result of this section---a subquadratic bound on $r_<(M, K_3)$ for random matchings with interval chromatic number $2$---as a corollary.

\begin{theorem}\label{thm:rmatchings}
Let $M$ be an ordered matching on $2m$ vertices with interval chromatic number $2$, picked uniformly at random. Then there is a constant $c$ such that $$r_<(M, K_3) \leq cm^{24/13}$$ with high probability.
\end{theorem}

\begin{proof}
Setting $\delta = 4/\log m$ the statement of Lemma~\ref{lemma:ordint} becomes $$\Pr_{\pi \in S_m} \left[\Int(\pi, \pi+h) \geq e^4 m^{2/3} \right] \leq e^{-e^4 m^{2/3}}.$$ Picking a matching $M$ on $2m$ vertices with interval chromatic number $2$ uniformly at random, we have $\Int(\pi(M), \pi(M)+h) \leq m^{2/3 + 4/\log m}$ for all $h \in [m]$ with high probability. Thus we can apply Theorem~\ref{theorem:cliquematching} with parameters $\epsilon = 1/3 - 4/\log m$ and $\alpha = \beta = 1/24 - 1/(2 \log m)$ and $n = cm^{24/13}$, where $c$ is chosen sufficiently large that $$n^{13/24 - 1/(2 \log m)}/4 - n^{1/12 - 1/\log m}/2 \geq 2m$$ and $$2n^{13/24 - 1/(2 \log m)} \geq 2m.$$ So with high probability, every bicoloring of $[2n]$ contains either a blue triangle or a red copy of $M$ or a red clique of size at least $2m$.
\end{proof}


\section{Future Work}\label{section:conclusion}

Many open questions about the ordered Ramsey numbers of matchings remain. Most significant, perhaps, is the original question posed by Conlon, Fox, Lee, and Sudakov: does there exist some $\epsilon > 0$ such that $r_<(M, K_3) \leq n^{2-\epsilon}$ for every ordered matching $M$ on $n$ vertices? Based on our Theorem \ref{thm:rmatchings}, a number of natural intermediate questions arise. In particular, a reasonably modest step beyond random matchings with $\chi_<(M) = 2$ would be the following:

\begin{conjecture}
For every  $\chi$, there is a constant $\epsilon(\chi) > 0$ such that
$$r_<(M, K_3) \leq O(n^{2-\epsilon(\chi)})$$
for almost every ordered matching $M$ on $n$ vertices with interval chromatic number $\chi_<(M) = \chi$.
\end{conjecture}

Conversely, we are curious how far from the truth the exponent $\frac{24}{13}$ in our Theorem \ref{thm:rmatchings} is. It seems plausible that our argument can be optimized to produce a significantly better bound, and we do not know of any lower bounds for this class of matchings that come anywhere near this bound.

Regarding parenthesis matchings, we were unable to find a family for which $r_<(M, K_3)$ is superlinear, leaving a slight gap beneath our upper bound. Such a construction would be quite interesting to us. 

\makeatletter
\renewcommand\subsubsection{\@startsection{subsubsection}{3}{\z@}%
                                     {-3.25ex\@plus -1ex \@minus -.2ex}%
                                     {-1.5ex \@plus -.2ex}
                                     {\normalfont\normalsize\bfseries}}
\makeatother

\subsubsection*{Acknowledgments.} This research was done through the MIT Summer Program in Undergraduate Research. I would like to thank Asaf Ferber for suggesting that I study ordered Ramsey numbers. I'd also like to thank Ankur Moitra and Davesh Maulik for their advice and support. I'd like to thank Nikhil Reddy for suggesting a key idea in the proof of Proposition~\ref{prop:nest}. Finally, I'm very grateful to my mentor Jake Wellens for his guidance and endless patience; this work could not have happened without him.

\appendix

\section{Convexity Inequalities}\label{section:inequalities}

We provide here proofs of Lemma~\ref{lemma:convex1} and Lemma~\ref{lemma:convex2}.

\theoremstyle{lemma}
\newtheorem*{lemma:convex1}{Lemma \ref{lemma:convex1}}
\begin{lemma:convex1}
Let $a_0, a_1, a_2, \dots,a_k \geq 0$ and $\delta > 1$ and $m > 0$ be real numbers. Let $r = {m}^{-1/(\delta-1)}$. If $s = \sum_{i=0}^k a_i \geq 1$ and $a_i \leq rs$ for all $1 \leq i \leq k$, then $$m(a_0 + ca_1^\delta + \dots + ca_k^\delta) \leq cs^\delta$$ for any $c \geq m$.
\end{lemma:convex1}

\begin{proof}
Suppose that $0 < a_i \leq a_j < rs$ for some distinct indices $1 \leq i,j \leq k$. Since $f(x) = x^\delta$ is a convex function, if we decrease $a_i$ and increase $a_j$ by a common amount $\min(a_i, rs - a_j)$, the left-hand side of the inequality increases, while the right-hand side remains constant. Furthermore, the number of values $a_i$ which are equal to neither $0$ nor $rs$ decreases. Hence, it suffices to prove the inequality in the case where no two such values exist. Without loss of generality, we have $a_1 = \dots = a_{n-1} = rs$ and $a_{n+1} = \dots = a_k = 0$. Observe that $n-1 = (s - a_0 - a_n)/(rs)$.

Now we have
\begin{align*}
m(a_0 + ca_1^\delta + \dots + ca_k^\delta)
&= ma_0 + mc(n-1)(rs)^\delta + mca_n^\delta \\
&= ma_0 + mc(s - a_0 - a_n)(rs)^{\delta-1} + mca_n^\delta \\
&\leq ma_0 + mc(s - a_0)(rs)^{\delta-1} \\
&\leq ca_0 + c(s - a_0)s^{\delta-1} \\
&\leq cs^\delta
\end{align*}
where the first inequality holds since $a_n \leq rs$, so $mca_n^\delta \leq mca_n(rs)^{\delta-1}$; the second inequality holds by the assumptions $c\geq m$ and $r = m^{-1/(\delta-1)}$; and the third inequality holds since $s^{\delta-1} \geq 1$.
\end{proof}

\theoremstyle{lemma}
\newtheorem*{lemma:convex2}{Lemma \ref{lemma:convex2}}
\begin{lemma:convex2}
Let $a_1, \dots, a_k \geq 0$ and $\delta \geq 1$ be real numbers. Let $r \in (0,1)$. If $s = \sum_{i=1}^k a_i$ and $a_i \leq rs$ for all $1 \leq i \leq k$, then $$a_1^\delta + \dots + a_k^\delta \leq r^{\delta-1}s^\delta.$$
\end{lemma:convex2}

\begin{proof}
As in the previous lemma, we only need to prove the case where $a_1 = \dots + a_{n-1} = rs$ and $a_{n+1} = \dots = a_k = 0$, since all other cases can be ``sharpened'' into this one. As before but dropping the $a_0$-term, $n-1 = (s - a_n)/(rs)$. The bound is now simple:
\begin{align*}
a_1^\delta + \dots + a_k^\delta
&= (n-1)(rs)^\delta + a_n^\delta \\
&= (s-a_n)(rs)^{\delta-1} + a_n^\delta \\
&\leq r^{\delta-1}s^\delta. \qedhere
\end{align*}
\end{proof}

\end{document}